\documentclass[12pt]{amsart}
\usepackage{amsthm}
\usepackage{amssymb}
\usepackage{amsmath}
\theoremstyle{plain}
\newtheorem{proposition}{Proposition}
\newtheorem{theorem}[proposition]{Theorem}
\newtheorem{lemma}[proposition]{Lemma}
\newtheorem{corollary}[proposition]{Corollary}

\theoremstyle{definition}
\newtheorem{definition}[proposition]{Definition}

\theoremstyle{definition}

\newtheorem{remark}[proposition]{Remark}

\numberwithin{equation}{section}

\numberwithin{proposition}{section}

\gdef\myletter{}

\let\savetheequation\theequation
\def\theequation{\savetheequation\myletter}

\newcommand{\CC}{{\mathbb C}}

\newcommand{\RR}{{\mathbb R}}

\newcommand{\PP}{{\mathbb P}}

\renewcommand{\date}{\today}
\def \bar{\overline}

\begin{document}

\vskip 3mm

\title[Pluripotential Energy]{\bf Pluripotential Energy}

\author{T. Bloom* and N. Levenberg}{\thanks{*Supported in part by NSERC grant}}

\address{University of Toronto, Toronto, Ontario M5S 2E4 Canada}  
\email{bloom@math.toronto.edu}

\address{Indiana University, Bloomington, IN 47405 USA}

\email{nlevenbe@indiana.edu}

\date

\maketitle
\section{Introduction.}\label{sec:intro}

In \cite{[BBGZ]}, the authors describe a variational approach to the complex Monge-Amp\`ere equation. In addition, they define a notion of an {\it electrostatic energy} $E^*(\mu)$ associated with a probability measure $\mu$ on a compact K\"ahler manifold $X$ of dimension $n$ (\cite{[BBGZ]}, Definition 4.3). Moreover, given a compact subset $H\subset X$ and a continuous function $v$ on $H$, they show that the {\it equilibrium measure} $\mu_{eq}(H,v)$, which is the Monge-Amp\`ere measure associated to the equilibrium weight of $(K,v)$, is the unique minimizer of the functional 
\begin{equation}\label{bbgzfunc} {\mathcal I}_v(\mu):=  E^*(\mu) +\int_H v d\mu
\end{equation}
over all probability measures $\mu$ on $H$. The functional in (\ref{bbgzfunc}) can thus be considered as a {\it weighted electrostatic energy} of a probability measure $\mu$ on $H$ and their result gives a beautiful analogue of the classical weighted logarithmic energy minimization for weights and compact sets in the complex plane. We will elaborate on this in Section \ref{sec:finalcom}.

We specialize to the situation where $X=\PP^n$, complex $n-$dimensional projective space. For probability measures $\mu$ on compact subsets of $\CC^n \subset \PP^n$ we define two functionals $J(\mu)$ and $W(\mu)$. Both involve discrete approximations to $\mu$ and multivariate Vandermonde determinants, and each has a clear one-variable analogue. Theorem \ref{overthm} implicitly shows that $J(\mu)$, defined via $L^2-$type approximation, is related to $\exp(-E^*(\mu))$. This is made explicit in Corollary \ref{aftergap}. We also show that $W(\mu)$, which uses $L^{\infty}-$type approximation, coincides with $J(\mu)$. These functionals give a direct interpretation of $E^*(\mu)$, which we prefer to call the {\it pluripotential energy} of $\mu$. For a continuous function $Q$ on a compact subset $H\subset \CC^n$ we define weighted versions of these functionals, $J^Q(\mu)$ and $W^Q(\mu)$, related to (\ref{bbgzfunc}) (Corollary \ref{finalcor} and equation (\ref{wtden})). We remark that it is essential to consider the weighted versions of these functionals even if one only wants to prove unweighted results. 

One of the origins for the approach used in this paper lies in the large deviation results for the empirical measure of certain random matrix ensembles (a good general reference on the subject of random matrices is \cite{AGZ}). The joint probability distribution of the eigenvalues in the Gaussian unitary ensemble (GUE) is the (square of a) weighted Vandermonde determinant in one variable, normalized to give a probability measure. The large deviation result expresses the asymptotic value (as $d\to \infty$) of the average of the joint probability distribution of a point ${\bf a}=(a_1,...,a_d)\in H^d$ for a compact set $H\subset \CC$ as the discrete measures $\frac{1}{d}\sum_{j=1}^d \delta(a_j)$ approach a fixed probability measure $\mu$ in $H$.

G. Ben Arous and A. Guionnet \cite{BeG}, building on work of Voiculescu, first gave a large deviation result for the GUE. Subsequently it was extended to any unitary invariant ensemble (equivalently, any (square of a) weighted Vandermonde determinant in one variable) (cf., \cite{AGZ}). The essential part of the rate function for the large deviation results, i.e., the part which is nonlinear in $\mu$, when normalized, is the (negative of the) logarithmic energy of the planar measure $\mu$. In this context it was called the {\it entropy} of the measure by Voiculescu. In effect, in this paper as well as in \cite{[BBGZ]} and \cite{[Ber]}, one starts from this point of view of logarithmic energy of planar measures in order to develop  a concept of energy of a measure in several variables.

 In \cite{[Ber]} Berman proves a large deviation result in the setting of determinantal point processes on a complex manifold $X$ with an appropriate Hermitian line bundle. The case $X=\PP^n$ and the hyperplane bundle over $X$ is the setting of pluripotential theory in $\CC^n$. Berman uses deep results from his work with Boucksom (cf., \cite{[BB]} and \cite{BBnew}) and results in \cite{[BBGZ]} and \cite{[BBN]}. In particular, the rate functional is given in terms of $E^*(\mu)$.

Our approach in this paper is somewhat different (although we certainly also use deep results from 
\cite{[BB]}, \cite{BBnew}, \cite{[BBGZ]} and \cite{[BBN]}). We follow the outline of the one-variable paper \cite{bloomvoic}; this outline was conjectured to work in the higher-dimensional case as well in \cite{btalk}. Section \ref{sec:back} gives the definitions of our functionals along with background material from (weighted) pluripotential theory and states our main result, Theorem \ref{overthm}. Elementary properties of our functionals $J(\mu)$ and $W(\mu)$, including a simple upper bound, are proved in Section \ref{sec:basic}. The more difficult lower bound, and the proof of Theorem \ref{overthm} for measures with finite pluripotential energy, is dealt with in Section \ref{sec:lbjmu}. Briefly, Markov's polynomial inequality, together with the equidistribution result of \cite{[BBN]} on Fekete points, is used to establish the lower bound in the case $\mu$ is a weighted equilibrium measure and then approximation arguments are used in the general case. Finally, we relate $J(\mu)=W(\mu)$ with $\exp (-E^*(\mu))$ for all probability measures $\mu$ and we give some final remarks in Sections \ref{sec:app} and \ref{sec:finalcom}. 

We are deeply indebted to S. Dinew for the proof of Proposition \ref{domprin}.

\section{Background and main result.}\label{sec:back}

Let $s=s(d):= {d+n \choose d}$ which is the dimension of the (complex) vector space of holomorphic polynomials ${\mathcal P}_d$ of degree at most $d$ in $\CC^n$. We fix a standard basis of monomials $\{e_1,...,e_s\}$ for ${\mathcal P}_d$ and given $s$ points $a_1,...,a_s\in \CC^d$, we write
$$VDM_d({\bf a})=VDM_d(a_1,...,a_s)=\det [e_i(a_j)]_{i,j=1,...,s}.$$
This is a polynomial of degree $l_d:=\sum_{j=1}^s \hbox{deg}e_j= \frac{n}{n+1}ds$ in $a_1,...,a_s$. More generally, given a nonnegative function $w:=e^{-Q}$, we write
\begin{equation} \label{wtdvdm} VDM^Q_d({\bf a})=VDM_d(a_1,...,a_s)w(a_1)^d \cdots w(a_s)^d.\end{equation}

For any compact set $H\subset \CC^n$, we have that
$$\lim_{d\to \infty} \bigl( \max_{{\bf a}\in H^s} |VDM_d({\bf a})|\bigr)^{1/l_d}=: \delta (H)$$
exists \cite{zah}; this is called the {\it transfinite diameter of $H$}. More generally, given an admissible weight $w=e^{-Q}$ on $H$, i.e., $w$ is uppersemicontinuous (usc) and $\{z\in H: w(z)>0\}$ is nonpluripolar, 
$$\lim_{d\to \infty} \bigl( \max_{{\bf a}\in H^s} |VDM^Q_d({\bf a})|\bigr)^{1/l_d}=: \delta^w (H)$$
exists \cite{[BL]}; this is called the {\it weighted transfinite diameter of $H,w$}. For simplicity, we define the slightly modified versions
$$\bar \delta (H):= \delta (H)^{\frac{n}{n+1}} \ \hbox{and} \ \bar \delta^w (H):= \delta^w (H)^{\frac{n}{n+1}}.$$

Given a compact set $H\subset \CC^n$ and a measure $\nu$ on $H$, the pair $(H,\nu)$ satisfies a {\it Bernstein-Markov property} if for all $p_d\in \mathcal P_d$, 	
\begin{equation}\label{bernmark}||p_d||_H\leq  M_d||p_d||_{L^2(\nu)}  \ \hbox{with} \ \limsup_{d\to \infty} M_d^{1/d} =1.\end{equation}
Then we have, for $(H,\nu)$ satisfying a Bernstein-Markov property (\ref{bernmark}),
$$\lim_{d\to \infty} [\int_{H^s} |VDM_d({\bf a})|^2d\nu({\bf a})]^{1/2l_d}= \delta (H)$$
(cf., \cite{[BL]}). Here we use the shorthand notation 
$$d\nu({\bf a}):= d\nu(a_1) \cdots d\nu(a_s)$$
for the product measure on $H^s$.

Let ${\mathcal M}(H)$ denote the space of probability measures on $H$. The weak-* topology on ${\mathcal M}(H)$ is given as follows . A 
neighborhood basis of any $\mu \in {\mathcal M}(H)$ is given by sets of the form 
$$\{\nu \in {\mathcal M}(H):
|\int_H f _i (d\mu - d\nu )| < \epsilon \ \hbox{for} \  i = 1,..., k\}$$ 
where $\epsilon >0$  and $f_1,...,f_k$ are continuous functions on $H$. We note that 
${\mathcal M}(H)$ is a complete metrizable space and a neighborhood basis 
of $\mu \in {\mathcal M}(H)$ is given by sets of the form 
$$G(\mu, k, \epsilon) := \{\nu  \in {\mathcal M}(H):
|\int_H x^{\alpha}y^{\beta} (d\mu - d\nu )| < \epsilon \ \hbox{for} \ |\alpha| +|\beta|\leq k\};$$
i.e., all probability measures on $H$ whose (real) moments, 
up to order $k$, are within $\epsilon$ of the corresponding moments for $\mu$. 

We write $L(\CC^n)$ for the set of all plurisubharmonic (psh) functions $u$ on $\CC^n$ with the property that $u(z) - \log |z| = 0(1), \ |z| \to \infty$ and
$$ L^+(\CC^n)=\{u\in L(\CC^n): u(z)\geq \log^+|z| + C\}$$
where $C$ is a constant depending on $u$. For $H\subset \CC^n$ compact and an admissible weight function $w=e^{-Q}$ on $H$, we define the {\it weighted pluricomplex Green function} $V^*_{H,Q}(z):=\limsup_{\zeta \to z}V_{H,Q}(\zeta)$ where
$$V_{H,Q}(z):=\sup \{u(z):u\in L(\CC^n), \ u\leq Q \ \hbox{on} \ H\}. $$
The case $w\equiv 1$ on $H$; i.e., $Q\equiv 0$, is the ``unweighted'' case and we simply write $V_H$. We have $V^*_{H,Q}\in L^+(\CC^n)$ and we call the complex Monge-Amp\`ere measure 
$$\mu_{eq}(H,Q):=(dd^cV_{H,Q}^*)^n$$
the {\it weighted equilibrium measure}; if $Q\equiv 0$ we write $\mu_{eq}(H)=(dd^cV_{H}^*)^n$. An example which we use later is the case where $T=\{(z\in \CC^n: |z_1|=\cdots =|z_n|=1\}$ is the unit torus; then 
\begin{equation}\label{torus} V_T(z) =\max_{j=1,...,n}\max [\log |z_j|, 0]. \end{equation} We say $H$ is locally regular if for each $z\in H$ the unweighted pluricomplex Green function for the sets $H\cap \overline{ B(z,r)}, \ r>0$ are continuous. Here $B(z,r)$ denotes the Euclidean ball with center $z$ and radius $r$. If $H$ is locally regular and $Q$ is continuous, then $V_{H,Q}$ is continuous (cf, \cite{sic}). We will use the elementary fact that for such $H$ and $Q$, 
\begin{equation} \label{suppw} V_{H,Q}(z)=V_{H,Q}^*(z)\leq Q(z) \ \hbox{on} \ H. \end{equation}
In general, it is known that 
$$\hbox{supp}(\mu_{eq}(H,Q))\subset \{z\in H: V_{H,Q}^*(z)\geq Q(z)\}$$
and that $V_{H,Q}^*=Q$ on $\hbox{supp}(\mu_{eq}(H,Q))$ except perhaps a pluripolar set (cf., \cite{ST}, Appendix B).

We remark that for locally bounded psh functions, e.g., for $u\in L^+(\CC^n)$, $(dd^cu)^n$ is well-defined as a positive measure. For simplicity, we choose our definition of $dd^c=\frac{i}{\pi}\partial \bar \partial$ so that 
$$\int_{\CC^n} (dd^cu)^n=1 \ \hbox{for all} \ u\in L^+(\CC^n).$$
For arbitrary $u\in L(\CC^n)$ one can define the {\it nonpluripolar Monge-Amp\`ere measure} as the weak-* limit
$$NP(dd^cu)^n:=\lim_{j\to \infty} \bigl( {\bf 1}_{\{u >-j\}}\cdot (dd^c\max[u, -j]) ^n\bigr)$$
(cf., \cite{bedtay}). For $u\in L(\CC^n)$ with \begin{equation} \label{nonpp} \int_{\CC^n} NP(dd^cu)^n=1,\end{equation}
we write $(dd^cu)^n:= NP(dd^cu)^n$. We only consider $u\in L(\CC^n)$ satisfying (\ref{nonpp}) in this paper.  

Our setting will be as follows. We let $H$ be a nonpluripolar compact set in $\CC^n$ with $V_H$ continuous and we fix a positive measure $\nu$ on $H$ with total mass at most one. We will assume $(H,\nu)$ satisfies a density property: there exists $T>0$ so that
\begin{equation} \label{massdens}\nu(B(z_0,r))\geq r^T\end{equation}
for all $z_0\in H$ and all $r<r(z_0)$ where $r(z_0)>0$. The density hypothesis implies that $(H,\nu)$ satisfies a Bernstein-Markov property (\ref{bernmark}); cf., \cite{bloom}, \cite{massdens} and the end of Section \ref{sec:finalcom}.

Given $\mu \in {\mathcal M}(H)$, and given a fixed neighborhood $G$ of $\mu$ in ${\mathcal M}(H)$, for each $s=1,2,...$ we set
$$\tilde G_s(\mu)=\tilde G_s(\mu,H):= \{{\bf a} \in H^s: \frac{1}{s}\sum_{j=1}^s \delta (a_j)\in G\}.$$
Technically, $\tilde G_s(\mu)$ depends on $G$, but not $\mu$; however, to emphasize that we begin with the measure $\mu$ in constructing these neighborhoods $G$, we maintain the 
notation. Define
$$J_d(\mu,G):=[\int_{\tilde G_s(\mu)}|VDM_d({\bf a})|^2d\nu ({\bf a})]^{1/2ds}$$
and
$$W_d(\mu,G):=\sup \{ |VDM_d({\bf a})|^{1/ds}: {\bf a} \in \tilde G_s(\mu)\}.$$

\begin{definition} \label{jwmu} We define
$$J(\mu):=\inf_{G \ni \mu} J(\mu,G) \ \hbox{where} \ J(\mu,G):=\limsup_{d\to \infty} J_d(\mu,G)$$
and
$$W(\mu):=\inf_{G \ni \mu} W(\mu,G) \ \hbox{where} \ W(\mu,G):=\limsup_{d\to \infty} W_d(\mu,G)\}.$$\end{definition}

\noindent Here the infimum is taken over all neighborhoods $G$ of the measure $\mu$ in ${\mathcal M}(H)$. 

Our main result is the following:

\begin{theorem}\label{overthm} Let $H$ be a nonpluripolar, compact, convex set. For $\mu \in {\mathcal M}(H)$ we have 
\begin{equation} \label{ell}\log J(\mu)=\log W(\mu)= \inf_{w} [\log \bar \delta^w (H) +\int_H Qd\mu]\end{equation}
where the infimum is taken over all  $w=e^{-Q}>0$ continuous on $H$.
\end{theorem}

The weighted versions of our functionals are defined starting with
$$J^Q_d(\mu,G):=[\int_{\tilde G_s(\mu)}|VDM^Q_d({\bf a})|^2d\nu ({\bf a})]^{1/2ds}$$
and
$$W^Q_d(\mu,G):=\sup \{ |VDM^Q_d({\bf a})|^{1/ds}: {\bf a} \in \tilde G_s(\mu)\}.$$

\begin{definition} \label{jwmuq} We define
$$J^Q(\mu):=\inf_{G \ni \mu} J^Q(\mu,G) \ \hbox{where} \ J^Q(\mu,G):=\limsup_{d\to \infty} J^Q_d(\mu,G)$$
and
$$W^Q(\mu):=\inf_{G \ni \mu} W^Q(\mu,G) \ \hbox{where} \ W^Q(\mu,G):=\limsup_{d\to \infty} W^Q_d(\mu,G).$$\end{definition}

\begin{corollary} \label{overcor} For $w=e^{-Q}>0$ a continuous weight function on a nonpluripolar, compact, convex set $H$, and for $\mu \in {\mathcal M}(H)$ we have
$$\log J^Q(\mu)=\log W^Q(\mu)=\inf_{\tilde w} [\log \bar \delta^{\tilde w (H)} +\int_H \tilde Qd\mu] -\int_H Q d\mu$$
where the infimum is taken over all  $\tilde w=e^{-\tilde Q}>0$ continuous on $H$.
\end{corollary}

We also show that if $\mu:=\mu_{eq}(H,\tilde Q)$ where $\tilde w=e^{-\tilde Q}>0$ is continuous, we can replace the ``$\limsup_{d\to \infty}$'' in Definitions \ref{jwmu} and \ref{jwmuq} by ``$\liminf_{d\to \infty}$'' thus by ``$\lim_{d\to \infty}$.''

\section{Elementary properties.}\label{sec:basic}

We prove some elementary properties of the functionals in Definitions \ref{jwmu} and \ref{jwmuq}. In this section $H$ need not be convex. Our first observation is that each functional is uppersemicontinous on ${\mathcal M}(H)$. Indeed, we prove the following.

 \begin{lemma} \label{usc}The functionals
$$\mu \to \inf_{G\ni \mu} \limsup_{d\to \infty} \log J_d(\mu,G)$$
and
$$\mu \to \inf_{G\ni \mu} \liminf_{d\to \infty} \log J_d(\mu,G)$$
defined for $\mu \in {\mathcal M}(H)$ are usc; analogous statements hold for the functionals $J^Q_d(\mu,G), \ W_d(\mu,G)$ and $W^Q_d(\mu,G)$.
\end{lemma}

\begin{proof} Given $\mu \in {\mathcal M}(H)$, take a sequence of measures $\mu_n$ in ${\mathcal M}(H)$ converging weak-* to $\mu$ and fix a neighborhood $G\subset {\mathcal M}(H)$ of $\mu$. For $n$ sufficiently large, let $G_n\subset {\mathcal M}(H)$ be a neighborhood of $\mu_n$ with $G_n \subset G$. Then
$$\log J_d(\mu_n,G_n) \leq \log J_d(\mu,G)$$ and hence
$$ \limsup_{d\to \infty} \log J_d(\mu_n,G_n) \leq \limsup_{d\to \infty} \log J_d(\mu,G);$$
thus for $n$ sufficiently large
$$\inf_{G_n\ni \mu_n}\limsup_{d\to \infty} \log J_d(\mu_n,G_n) \leq \limsup_{d\to \infty} \log J_d(\mu,G).$$
We conclude that
$$\limsup_{n\to \infty}\inf_{G_n\ni \mu_n}\limsup_{d\to \infty} \log J_d(\mu_n,G_n) \leq \limsup_{d\to \infty} \log J_d(\mu,G)$$
for all neighborhoods $G$ of $\mu$; thus
$$\limsup_{n\to \infty}\inf_{G_n\ni \mu_n}\limsup_{d\to \infty} \log J_d(\mu_n,G_n) \leq  \inf_{G \ni \mu} \limsup_{d\to \infty} \log J_d(\mu,G)$$
which is the desired (first) result. A similar proof works for the functional $\mu \to \inf_{G \ni \mu} \liminf_{d\to \infty} \log J_d(\mu,G)$.
\end{proof}

Some elementary inequalities follow from the definitions.

\begin{lemma} \label{easylemma} Let $\mu \in {\mathcal M}(H)$ and let $w=e^{-Q}$ be admissible  on $H$. Then
$$J^Q(\mu)\leq W^Q(\mu)\leq \bar \delta^w (H) .$$
\end{lemma}

\begin{proof} We have, since $\nu(H)\leq 1$,
$$J^Q_d(\mu,G):=[\int_{\tilde G_s(\mu)}|VDM^Q_d({\bf a})|^2d\nu ({\bf a})]^{1/2ds}$$
$$\leq W^Q_d(\mu,G):=\sup \{ |VDM^Q_d({\bf a})|^{1/ds}: {\bf a} \in \tilde G_s(\mu)\}$$
$$\leq \sup \{ |VDM^Q_d({\bf a})|^{1/ds}: {\bf a} \in H^s\}.$$
From the definition of the weighted transfinite diameter, 
$$\lim_{d\to \infty}\sup \{ |VDM^Q_d({\bf a})|^{1/ds}: {\bf a} \in H^s\}=\bar \delta^w (H)$$
and the result follows.
\end{proof}

To prove Corollary \ref{overcor} from Theorem \ref{overthm}, and to get an upper bound on $J(\mu)$ and $W(\mu)$, we begin with a lemma.

\begin{lemma} \label{weakconv}Let $\mu \in {\mathcal M}(H)$ and let $Q$ be continuous on $H$. Given $\epsilon >0$, there exists a neighborhood $G \subset  {\mathcal M}(H)$ of $\mu$ with
$$|\int_H Q(d\mu -d\tilde \mu)|< \epsilon \ \hbox{for} \ \tilde \mu \in G.$$
\end{lemma}

\begin{corollary} \label{weakconvcor} Let $\mu \in {\mathcal M}(H)$ and let $Q$ be continuous on $H$. Given $\epsilon >0$, there exists a neighborhood $G \subset  {\mathcal M}(H)$ of $\mu$ with
$$|\int_H Q(d\mu -\frac{1}{s}\sum_{j=1}^s\delta(a_j))|< \epsilon \ \hbox{for} \ {\bf a} \in \tilde G_s(\mu)$$
for $s=1,2,...$.
\end{corollary}

To prove Lemma \ref{weakconv}, we assume the conclusion is false; hence we get an $\epsilon >0$ and a sequence of measures $\mu_n$ in ${\mathcal M}(H)$ converging weak-* to $\mu$ with
$$|\int_H Q(d\mu -d\mu_n)|\geq  \epsilon \ \hbox{for all} \ n.$$
But this contradicts the weak-* convergence since $Q$ is continuous. 

Rewriting the conclusion of Corollary \ref{weakconvcor}, we have
$$-\epsilon \leq \int_H Q d\mu-\frac{1}{s}\sum_{j=1}^sQ(a_j)\leq \epsilon.$$
Thus for $w=e^{-Q}>0$,
\begin{equation}
\label{twelve}
e^{-sd\epsilon}\leq w(a_1)^d \cdots w(a_s)^d \bigl(e^{\int Qd\mu}\bigr)^{sd} \leq  e^{sd\epsilon}.\end{equation}
Now we use (\ref{twelve}) to prove the following relationship between the unweighted and weighted functionals, which immediately yields Corollary \ref{overcor} from  Theorem \ref{overthm}.

\begin{proposition} \label{easyprop} Let $\mu \in {\mathcal M}(H)$ and let $w=e^{-Q}>0$ be continuous on $H$. Then
$$J(\mu)=J^Q(\mu)\cdot e^{\int_H Qd\mu} \ \hbox{and} \  W(\mu)=W^Q(\mu)\cdot e^{\int_H Qd\mu}.$$
\end{proposition}

\begin{proof} Using (\ref{twelve}) and (\ref{wtdvdm}), given $\epsilon >0$, for ${\bf a}\in \tilde G_s(\mu)$, 
$$|VDM_d({\bf a})|e^{ds(-\epsilon-\int_H Qd\mu)} \leq |VDM^Q_d({\bf a})| \leq |VDM_d({\bf a})|e^{ds(\epsilon-\int_H Qd\mu)}.$$
Hence
$$|VDM^Q_d({\bf a})|e^{ds(-\epsilon+\int_H Qd\mu)} \leq |VDM_d({\bf a})| \leq |VDM^Q_d({\bf a})|e^{ds(\epsilon+\int_H Qd\mu)}.$$
Now we take the supremum over ${\bf a}\in \tilde G_s(\mu)$ and take $ds-$th roots of each side to get
\begin{equation}\label{jclone}W^Q_d(\mu,G)e^{-\epsilon}e^{\int_H Q d\mu} \leq W_d(\mu,G) \leq W^Q_d(\mu,G)e^{\epsilon}e^{\int_H Q d\mu}.\end{equation}
Precisely, given $\epsilon >0$, these inequalities are valid for $G$ a sufficiently small neighborhood of $\mu$. Hence we get, upon taking $\limsup_{d\to \infty}$, the infimum over $G\ni \mu$, and noting that $\epsilon >0$ is arbitrary,
$$W(\mu)=W^Q(\mu)\cdot e^{\int_H Qd\mu}$$
as desired. A similar proof shows that $J(\mu)=J^Q(\mu)\cdot e^{\int_H Qd\mu}$.
\end{proof}

We can now give a useful upper bound on $J(\mu)$ and $W(\mu)$.

\begin{proposition} \label{upper} For $\mu \in {\mathcal M}(H)$ we have 
$$\log J(\mu)\leq \log W(\mu)\leq \inf_w [\log \bar \delta^w (H) +\int_H Qd\mu]$$
where the infimum is taken over all continuous weights $w=e^{-Q}>0$ on $H$.
\end{proposition}

\begin{proof} Using $W(\mu)=W^Q(\mu)\cdot e^{\int_H Qd\mu}$ for any continuous weight $w=e^{-Q}>0$ on $H$ from Proposition \ref{easyprop}, Lemma \ref{easylemma} gives
$$\log J(\mu)\leq \log W(\mu)=\log W^Q(\mu)+ \int_H Qd\mu \leq \log \bar \delta^w(H) + \int_H Qd\mu.$$
\end{proof}

\section{Proof of Theorem \ref{overthm} if $J(\mu)=W(\mu)>0$.}\label{sec:lbjmu} 

In this section we prove Theorem \ref{overthm} in the case when $\mu \in {\mathcal M}(H)$ is such that there exists $u\in L(\CC^n)$ satisfying (\ref{nonpp}) with $(dd^cu)^n=\mu$ and $\int_H ud\mu > -\infty$. In Section \ref{sec:app}  we will see that  this happens precisely when $J(\mu)=W(\mu)>0$.

\begin{proof} The proof proceeds in several steps.
\vskip6pt

\noindent {\bf Step 1}: {\sl We prove the lower bound
$$\log J(\mu)\geq \inf_w [\log \bar \delta^w (H) +\int_H Qd\mu] >-\infty$$
where the infimum is taken over all  $w=e^{-Q}>0$ continuous on $H$ in the case $\mu = \mu_{eq}(H,Q):= (dd^cV^*_{H,Q})^n$ for some polynomial  weight $w=e^{-Q}>0$ on $H$; i.e., $w$ is a {\it real} polynomial in $\RR^{2n}$ with $w> 0$ on $H$. Hence (\ref{ell}) holds in this case.}
\vskip6pt

We begin with some preliminaries. First, since $H$ is assumed to be convex, considering $H$ as a subset of $\RR^{2n}$, $H$ satisfies a Markov inequality of exponent two for real polynomials: there exists $M=M(H)>0$ with
$$|\nabla p(x)| \leq M (\hbox{deg}p)^2 ||p||_H \ \hbox{for all} \ x\in H$$
for all real polynomials on $\RR^{2n}$ (cf., \cite{W}). Thus for any two points $x,y\in H$, integrating this inequality along the line segment in $H$ joining the points, we have
\begin{equation}\label{conxmark}|p(x)-p(y)|\leq M(\hbox{deg}p)^2 ||p||_H|x-y|.\end{equation}

\begin{lemma}\label{pospoly} Let $\Lambda_d({\bf a})=\Lambda_d(a_1,...,a_s)$ be a nonnegative polynomial on $H^s$, where $s =s(d):= {d+n \choose d}$, such that for each $j=1,...,s$, the polynomial 
$$t\to \Lambda_d(a_1,...,a_{j-1},t,a_{j+1},...,a_s)$$
is a polynomial of degree at most $\alpha d^{\beta}$ for positive constants $\alpha, \beta$ (independent of $d$). If ${\bf a}^*=(a_1^*,...,a_s^*)\in H^s$ satisfies 
$$\Lambda_d({\bf a}^*) = \max_{{\bf a}\in H^s}\Lambda_d({\bf a}),$$
then for all ${\bf a}\in \Delta_d({\bf a}^*):=\{{\bf a}\in H^s: |a_j -a_j^*| \leq e^{-\sqrt d}, \ j=1,...,s\}$,
$$\Lambda_d({\bf a}) \geq \Lambda_d({\bf a}^*)\bigl(1-\alpha' d^{\beta'}e^{-\sqrt d}\bigr)$$
for positive constants $\alpha', \beta'$ (independent of $d$).
\end{lemma}

\begin{proof} One writes, for ${\bf a}\in \Delta_d({\bf a}^*)$, 
$$\Lambda_d({\bf a}^*) - \Lambda_d({\bf a})=$$
$$\sum_{j=1}^s\bigl(  \Lambda_d(a_1^*,...,a_j^*,a_{j+1},...,a_s) -\Lambda_d(a_1^*,...,a_{j-1}^*,a_{j},...,a_s) \bigr)$$
and applies (\ref{conxmark}). Thus 
$$|\Lambda_d({\bf a}^*) - \Lambda_d({\bf a})|\leq sM\Lambda_d({\bf a}^*) [\alpha d^{\beta}]^2e^{-\sqrt d}$$
so that
$$\Lambda_d({\bf a})\geq \Lambda_d({\bf a}^*)\bigl( 1-  sM [\alpha d^{\beta}]^2e^{-\sqrt d} \bigr).$$
\end{proof}

\begin{remark} The lemma is valid more generally if $s\leq a d^b$ for positive constants $a, b$ (independent of $d$).
\end{remark}

We proceed with the proof of Step 1. Thus we let $w=e^{-Q}>0$ be a polynomial on $H$ of degree $k$ and we observe that for each $d=1,2,...$,
$${\bf a}\to |VDM_d^Q({\bf a})|^2$$
is a nonnegative polynomial on $H^s$ of degree at most $2kds +ds \leq \alpha d^{\beta}$. Choose ${\bf a}^{*,d} \in H^s$ to be a {\it weighted Fekete set of order $d$ for $H,Q$} (i.e., maximizing $ |VDM_d^Q({\bf a})|$ and hence $ |VDM_d^Q({\bf a})|^2$ over all $s-$tuples ${\bf a}$ in $H^s$). By \cite{[BBN]}, 
$$\frac{1}{s} \sum_{j=1}^s \delta(a_j^{*,d})\to  \mu_{eq}(H,Q) \ \hbox{weak-}*.$$
In particular, given a neighborhood $G$ of $\mu_{eq}(H,Q)$, for $d$ sufficiently large, $\frac{1}{s} \sum_{j=1}^s \delta(a_j^{*,d})\in G$. Moreover, by Proposition 1.3 of \cite{bloomvoic}, for $d$ sufficiently large, $\Delta_d({\bf a}^{*,d})\subset  G$. Thus, for such $d$,
$$J_d^Q(\mu_{eq}(H,Q),G)^{2ds}=\int_{\tilde G_s(\mu_{eq}(H,Q))} |VDM_d^Q({\bf a})|^2d\nu({\bf a}) $$
$$\geq \int_{\Delta_d({\bf a}^{*,d})} |VDM_d^Q({\bf a})|^2d\nu({\bf a}).$$
From Lemma \ref{pospoly} and (\ref{massdens}) this last integral satisfies
$$\int_{\Delta_d({\bf a}^{*,d})} |VDM_d^Q({\bf a})|^2d\nu({\bf a})\geq |VDM_d^Q({\bf a}^{*,d})|^2\cdot e^{-\sqrt d Ts}(1-\alpha' d^{\beta'}e^{-\sqrt d}).$$
We conclude that
$$\log J_d^Q(\mu_{eq}(H,Q),G)$$
$$\geq  \frac{1}{ds} \log |VDM_d^Q({\bf a}^{*,d})| -\frac{\sqrt d Ts}{2ds}+ \frac{1}{2ds}\log (1-\alpha' d^{\beta'}e^{-\sqrt d}).$$
Hence
$$\liminf_{d\to \infty} \log J_d^Q(\mu_{eq}(H,Q),G)\geq \log \bar \delta^w (H)$$
so that
\begin{equation}\label{ineq}\inf_{G \ni \mu_{eq}(H,Q)} \bigl(\liminf_{d\to \infty} \log J_d^Q(\mu_{eq}(H,Q),G)\bigr)\geq \log \bar \delta^w (H).\end{equation}
Using the version of (\ref{jclone}) for $J_d,J_d^Q$, we have, given $\epsilon>0$, for $G$ a sufficiently small neighborhood of $\mu_{eq}(H,Q)$,
$$J^Q_d(\mu_{eq}(H,Q),G)e^{-\epsilon}e^{\int_H Q d\mu_{eq}(H,Q)} \leq J_d(\mu_{eq}(H,Q),G). $$
Taking logarithms we conclude that 
$$\log J_d(\mu_{eq}(H,Q),G)\geq \log J^Q_d(\mu_{eq}(H,Q),G)-\epsilon+\int_H Q d\mu_{eq}(H,Q).$$
Hence, given $\epsilon>0$ and for $G$ a sufficiently small neighborhood of $\mu_{eq}(H,Q)$,
$$\liminf_{d\to \infty} \log J_d(\mu_{eq}(H,Q),G)$$
$$\geq \liminf_{d\to \infty} \log J^Q_d(\mu_{eq}(H,Q),G)-\epsilon+\int_H Q d\mu_{eq}(H,Q).$$
Since $\epsilon>0$ is arbitrary, from (\ref{ineq}) we have
$$\inf_G \bigl(\liminf_{d\to \infty}\log J_d(\mu_{eq}(H,Q),G)\bigr)\geq \log \bar \delta^w (H)+\int_H Q d\mu_{eq}(H,Q).$$
The right-hand-side is a candidate for the infimum in $\inf_{\tilde w} [\log \bar \delta^{\tilde w} (H) +\int_H \tilde Q\mu_{eq}(H,Q)]$; thus we have proved the lower bound in Step 1. Together with the upper bound in Proposition \ref{upper} we have shown that (\ref{ell}) holds in this case and we can replace ``limsup'' by limit in the definitions of $J,W$; i.e.,
$$ \inf_{G\ni \mu_{eq}(H,Q)} \bigl(\lim_{d\to \infty}\log J_d(\mu_{eq}(H,Q),G)\bigr)=\inf_{\tilde w} [\log \bar \delta^{\tilde w} (H) +\int_H \tilde Qd\mu_{eq}(H,Q)]$$
\begin{equation}\label{longeq}=\log \bar \delta^w (H)+\int_H Q d\mu_{eq}(H,Q)=\log J(\mu_{eq}(H,Q))>-\infty. \end{equation}
\vskip6pt

\noindent {\bf Step 2}: {\sl We prove the lower bound
$$\log J(\mu)\geq \inf_w [\log \bar \delta^w (H) +\int_H Qd\mu] >-\infty$$
where the infimum is taken over all  $w=e^{-Q}>0$ continuous on $H$ in the case $\mu = \mu_{eq}(H,Q):= (dd^cV^*_{H,Q})^n$ for some continuous weight $w=e^{-Q}>0$ on $H$. Hence (\ref{ell}) holds in this case.}
\vskip6pt

We take $\mu:=\mu_{eq}(H,Q)$ where $w=e^{-Q}>0$ is continuous. Let $\{w_n\}$ be a sequence of polynomials in $\RR^{2n}$ which converge uniformly to $w$ on $H$. Since $w> 0$ on $H$, we can take $w_n> 0$ on $H$ for $n$ sufficiently large. It follows that $V_{H,Q_n}\to V_{H,Q}$ uniformly on $\CC^n$ so that, in particular, 
$$\mu_n:=\mu_{eq}(H,Q_n)\to \mu \ \hbox{weak-}*.$$
By the uniform convergence of $Q_n$ to $Q$,
$$\lim_{n\to \infty} \int_H Q_n d\mu_n = \int_H Q d\mu.$$
We also have
$$\lim_{n\to \infty} \log \bar \delta^{w_n} (H) = \log \bar \delta^w (H).$$
Using the uppersemicontinuity in Lemma \ref{usc}, 
$$\inf_{G\ni \mu} \bigl(\liminf_{d\to \infty}\log J_d(\mu,G)\bigr)\geq \limsup_{n\to \infty} \inf_{G_n\ni \mu_n} \bigl(\liminf_{d\to \infty}\log J_d(\mu_n,G_n)\bigr).$$
By Step 1, the right-hand-side equals
$$\limsup_{n\to \infty} [\log \bar \delta^{w_n} (H)+\int_H Q_n d\mu_n]$$
which equals
$$\log \bar \delta^w (H)+\int_H Q d\mu >-\infty$$
from the above observations. This proves the lower bound. Together with the upper bound in Proposition \ref{upper} 
we have again shown that (\ref{ell}) holds in this case and we can replace ``limsup'' by limit (\ref{longeq}).

\vskip6pt

\noindent {\bf Step 3}: {\sl We prove the theorem if $\mu \in {\mathcal M}(H)$ is such that there exists $u\in L(\CC^n)$ satisfying (\ref{nonpp}) with $(dd^cu)^n=\mu$ and $\int_H ud\mu > -\infty$.}
\vskip6pt

We begin with the following.

\begin{proposition} \label{goodapprox} There exist continuous $u_j\in L^+(\CC^n)$ with $u_j \downarrow  u$ on $\CC^n$ and 
$\mu_j:=(dd^cu_j)^n$ supported in $H$. In particular, $\mu_j \to \mu =(dd^cu)^n$ weak-*. Moreover,
\begin{equation}\label{weak*}\lim_{j\to \infty} \int_{H} u_jd\mu_j = \lim_{j\to \infty} \int_{H} u_jd\mu=\int_H ud\mu>-\infty .\end{equation}
\end{proposition}

\begin{proof} We know there exist $v_j\in L^+(\CC^n)$ continuous (even smooth) with $v_j \downarrow u$. Define the weighted pluricomplex Green functions
$$u_j =V_{H,v_j|_{H}}:= \sup \{ U\in L(\CC^n): U \leq v_j \ \hbox{on} \  H\}.$$
Since $\{v_j\}$ are monotone, so are $\{u_j\}$ and clearly $u_j \geq v_j$. Moreover, since $H$ is locally regular and $v_j|_{H}$ is continuous, $u_j$ is continuous and $u_j \leq v_j$ on $H$. But each  $v_j$ is a competitor in the definition of $u_j =V_{H,v_j|_{H}}$ so that $v_j \leq u_j$ on $\CC^n$ and we have $u_j = v_j$ on $H$. Thus $\tilde u:= \lim_{j\to \infty} u_j  \geq u$ everywhere and $\tilde u = u$ on $H$. From the domination principle (Corollary \ref{domcorr} to be proved in Section \ref{sec:app}), $\tilde u = u$ on $\CC^n$. Relation (\ref{weak*}) follows from Theorem 2.1 of \cite{[CGZ]} and the fact that each measure $\mu_j$ is supported in $H$.
\end{proof}

\begin{remark} \label{decapprox} If $u\in L(\CC^n)$ and $(dd^cu)^n$ is not supported in $H$, the proof of Proposition \ref{goodapprox} yields the existence of continuous $u_j\in L^+(\CC^n)$ satisfying $u_j \downarrow \tilde u \geq u$ with $\tilde u =u$ on $H$ and $\mu_j:=(dd^cu_j)$ supported in $H$. This will be used in Section \ref{sec:app}.
\end{remark}

Now since $\mu_j$ is supported in $H$ and $u_j =V_{H,v_j|_{H}}$ we have, from Step 2 with $w_j=e^{-Q_j}$ where $Q_j=v_j|_{H}=u_j|_{H}$, that
$$\log W(\mu_j)=\log J(\mu_j) =\log \bar \delta^{w_j} (H)+ \int_{H} u_j d\mu_j.$$
Since $w_j \uparrow w$ on $H$ where $w=e^{-Q}$ with $Q=u|_{H}$, the limit 
$$\lim_{j\to \infty}\log \bar \delta^{w_j} (H)$$
exists and is finite; hence by (\ref{weak*}) 
$$\lim_{j\to \infty}\log W(\mu_j)=\lim_{j\to \infty}\log J(\mu_j)=\lim_{j\to \infty}[\log \bar \delta^{w_j} (H)+\int_{H} u_jd\mu_j] >-\infty.$$
For simplicity in notation, we work with the functional $W$. The uppersemicontinuity of the functional $W$ (Lemma \ref{usc}) gives 
$$\log W(\mu) \geq \lim_{j\to \infty}[\log \bar \delta^{w_j} (H)+ \int_{H} u_j d\mu_j].$$
In particular, given $\epsilon >0$, for $j>j_0(\epsilon)$ we have, from this and from (\ref{weak*}), 
\begin{equation} \label{fourfive}\log W(\mu) \geq \log W(\mu_j) -\epsilon =\log \bar \delta^{w_j} (H)+ \int_{H} u_j d\mu_j-\epsilon \end{equation}
$$\geq \log \bar \delta^{w_j} (H)+ \int_{H} u_j d\mu- 2\epsilon.$$ 
Since $w_j=e^{-Q_j}$ where $Q_j=u_j|_{H}$ is a continuous weight on $H$, 
$$\log \bar \delta^{w_j} (H)+ \int_{H} u_j d\mu \geq \inf_{\tilde w} [\log \bar \delta^{\tilde w} (H) +\int_{H} \tilde Qd\mu].$$
From the upper bound for $W(\mu)$ in Proposition \ref{upper},
$$\log W(\mu)\leq \inf_{\tilde w} [\log \bar \delta^{\tilde w} (H) +\int_{H} \tilde Qd\mu]\leq \log W(\mu_j) +\epsilon \leq \log W(\mu) +2\epsilon$$
and hence
$$ \log W(\mu)= \inf_{\tilde w} [\log \bar \delta^{\tilde w} (H) +\int_{H} \tilde Qd\mu]. $$
The same arguments apply using the functional $J$; hence we conclude that
\begin{equation} \label{wd} \log W(\mu)=\log J(\mu)= \inf_{\tilde w} [\log \bar \delta^{\tilde w} (H) +\int_{H} \tilde Qd\mu]. \end{equation}
\end{proof}

\begin{remark} \label{independ} In particular, since $J(\mu)=W(\mu)$, $J(\mu)$ is independent of the measure $\nu$, provided (\ref{massdens}) holds. Another observation is that the convexity of $H$ was used in Step 1 in order to have a line segment joining a pair of points in $H$ along which a Markov inequality could be integrated (equation (\ref{conxmark})) and in Proposition \ref{goodapprox} to conclude that $H$ is locally regular. If $H$ is a fat ($H=\overline{H^o}$) connected, subanalytic set, a similar proof holds using the facts that such an $H$ is locally regular \cite{Pl} and satisfies a Markov inequality with perhaps a different exponent and that pairs of points $x,y$ in $H$ can be joined by rectifiable arcs whose length is dominated, up to a universal constant, by a fixed power of the Euclidean distance from $x$ to $y$ since $H$ is Whitney $p-$regular for some $p$ (\cite{bm}, Definition 6.9 and Theorem 6.10).
\end{remark}

\section{Relation with $E^*(\mu)$ and the case $J(\mu)=W(\mu)=0$.}\label{sec:app}

We give a brief synopsis of a special case of \cite{[BBGZ]} and reconcile it with our results. A good discussion of the material can be found in \cite{[GZ]}, \cite{[GZ2]}, and \cite{[Ber]}. Let $X=\PP^n$ with the usual K\"ahler form $\omega$ normalized so that $\int_{\PP^n} \omega^n =1$. Define the class of {\it $\omega-$psh functions}
$$PSH(X,\omega) :=\{\phi \in L^1(X): \phi \ \hbox{usc}, \ dd^c\phi +\omega \geq 0\}.$$
For any $\phi \in PSH(X,\omega)$ one can define the {\it nonpluripolar Monge-Amp\`ere measure} as 
 $$MA(\phi):=\lim_{j\to \infty} \bigl( {\bf 1}_{\{\phi>-j\}}\cdot (dd^c\max[\phi, -j] +\omega)^n\bigr).$$

 \begin{definition} We write
 $$\mathcal E(X,\omega):=\{\phi \in PSH(X,\omega): \int_X MA(\phi)=1\}.$$
 \end{definition} 
 \noindent It is known that $\mathcal E(X,\omega)$ is convex (\cite{[GZ2]}, Proposition 1.6). For $\phi \in \mathcal E(X,\omega)$, we write $MA(\phi)= (dd^c \phi +\omega)^n$.

For bounded $\phi \in PSH(X,\omega)$, define
\begin{equation} \label{energeqn}E(\phi):=\frac{1}{n+1}\sum_{j=0}^n\int_X \phi (dd^c\phi +\omega)^j\wedge \omega^{n-j}\end{equation}
and extend $E$ to $PSH(X,\omega)$ via
 $$E(\psi)=\lim_{j\to \infty} E(\max[\psi, -j]).$$
This operator is {\it monotone} (cf., \cite{[BB]}, \cite{BBnew}): 
\begin{equation} \label{emono} u_1 \geq u_2 \ \hbox{implies} \ E(u_1) \geq E(u_2).\end{equation}
If $\phi \in \mathcal E(X,\omega)$, the nonpluripolar part of $(dd^c\phi +\omega)^j\wedge \omega^{n-j}$ has total mass one (\cite{[BEGZ]}, Corollary 2.15); indeed, formula (\ref{energeqn}), appropriately interpreted, makes sense for such $\phi$ (\cite{[BEGZ]}, Corollary 2.18).
 
 \begin{definition} We write
 $$\mathcal E^1(X,\omega):=\{\phi \in \mathcal E(X,\omega): E(\phi) > -\infty\}.$$
 \end{definition}

Let ${\mathcal M}(X)$ denote the probability measures on $X$. A result of Guedj and Zeriahi in \cite{[GZ2]} is the following.
 
 \begin{theorem} \label{monge} \cite{[GZ2]} Let $\mu\in {\mathcal M}(X)$. 
 \begin{enumerate}
 \item If $\mu(P)=0$ whenever $P\subset X$ is pluripolar, there is a $\phi \in \mathcal E(X,\omega)$ with $MA(\phi)=\mu$.
\item  If $\mathcal E^1(X,\omega)\subset L^1(\mu)$, there is a $\phi \in \mathcal E^1(X,\omega)$ with $MA(\phi)=\mu$.
\end{enumerate}
 \end{theorem}
 
We remark that an equivalent characterization of  $\mathcal E^1(X,\omega)$ is the following: {\it for $\phi \in PSH(X,\omega)$, we have $\phi \in \mathcal E^1(X,\omega)$ if and only if $\int_X MA(\phi) =1$ and $\int_X \phi MA(\phi) >-\infty$} (\cite{[BEGZ]}, Proposition 2.11). 
 
  In \cite{[BBGZ]}, a variational approach to Theorem \ref{monge} is given. Given $\mu \in {\mathcal M}(X)$, define a functional $F_{\mu}$ on $PSH(X,\omega)$ via
 $$F_{\mu}(\phi)=E(\phi)-\int_X \phi d\mu.$$
 
 \begin{theorem} \cite{[BBGZ]}
 \label{varthm}
 Given $\mu \in {\mathcal M}(X)$ with $\mathcal E^1(X,\omega)\subset L^1(\mu)$, 
 $$F_{\mu}(\phi) =\sup_{\psi \in \mathcal E^1(X,\omega)}F_{\mu}(\psi)$$ for some $\phi \in \mathcal E^1(X,\omega)$ if and only if $\mu =MA(\phi)$.
 \end{theorem}
 
Following \cite{[BBGZ]}, we define the {\it electrostatic energy of $\mu \in {\mathcal M}(X)$}:
 $$E^*(\mu):= \sup_{\phi \in  \mathcal E^1(X,\omega)} F_{\mu}(\phi)= \sup_{\phi \in \mathcal E^1(X,\omega)}[E(\phi)-\int_X \phi d\mu].$$
In terms of $E^*(\mu)$, Theorem \ref{varthm} yields (see also Proposition 2.5 of \cite{[Ber]}):

 \begin{corollary} \cite{[BBGZ]}
 \label{varprop} Given $\mu \in {\mathcal M}(X)$
\begin{equation} \label{char} E^*(\mu) <+\infty \ \hbox{if and only if} \ \mathcal E^1(X,\omega)\subset L^1(\mu), \end{equation} 
and, in this case, writing $\mu =MA(\phi)$ for $\phi \in \mathcal E^1(X,\omega)$, 
 \begin{equation}
\label{relation}
E^*(\mu)= E(\phi)-\int_X \phi d\mu.
\end{equation}
\end{corollary}

For $\mu  \in {\mathcal M}(H)$ where $H$ is a compact set in $\CC^n$, the connection between $E^*(\mu)$ and our functionals $J(\mu), \ W(\mu)$ begins with equation (2.3) in \cite{[Ber]}. Since this is a crucial result, we provide details. Let $[z_0:z_1:\cdots :z_n]$ be homogeneous coordinates on $X=\PP^n$. Identifying $\CC^n$ with the affine subset of $\PP^n$ given by $\{[1:z_1:\cdots:z_n]\}$, we can identify the $\omega-$psh functions with the Lelong class $L(\CC^n)$, i.e., 
$$PSH(X,\omega) \approx L(\CC^n),$$
and the bounded (from below) $\omega-$psh functions coincide with the subclass $L^+(\CC^n)$. Indeed, if $\phi \in PSH(X,\omega)$, then (cf., \cite{[GZ]}) 
$$u(z)=u(z_1,...,z_n):= \phi ([1:z_1:\cdots:z_n])+\frac{1}{2}\log (1+|z|^2)\in  L(\CC^n);$$
if $u\in  L(\CC^n)$, define $\phi \in PSH(X,\omega)$ via
$$\phi ([1:z_1:\cdots:z_n])=u(z)-\frac{1}{2}\log (1+|z|^2) \ \hbox{and}$$ 
$$\phi ([0:z_1:\cdots:z_n])=\limsup_{t\to \infty, \ t\in \CC}[u(tz)-\frac{1}{2}\log (1+|tz|^2)].$$ 
Abusing notation, we write $u= \phi +u_0$ where $u_0(z):=\frac{1}{2}\log (1+|z|^2)$. 

For $u\in L^+(\CC^n)$ we consider the functional 
 $$E(u):= \frac{1}{n+1}\sum_{j=0}^n\int_{\CC^n} (u-u_0) dd^cu^j\wedge dd^cu_0^{n-j}.$$
To justify our use of the same notation as in (\ref{energeqn}), note for $u\in L^+(\CC^n)$ and its associated bounded $\phi \in PSH(X,\omega)$, we have $E(u)=E(\phi)$. We extend the functional to $u= \phi +u_0\in L(\CC^n)$ using the canonical approximations $u_j := \max[\phi,-j]+u_0$ and observe that the relation $E(u)=E(\phi)$ still holds if $u\in L(\CC^n)$ satisfies (\ref{nonpp}) as the nonpluripolar part of each $dd^cu^j\wedge dd^cu_0^{n-j}$ has total mass one.
  
 Now suppose $\mu \in {\mathcal M}(H)$ with $H\subset \CC^n$ is such that there exists $u=\phi + u_0\in L(\CC^n)$ with $(dd^cu)^n=\mu$ and $\int_H ud\mu >-\infty$, our a priori assumption in Section \ref{sec:lbjmu}. Since $$ \int_H (dd^cu)^n =1 \ \hbox{implies} \ \int_H u_0(dd^cu)^n < +\infty$$
we have
$$\int_H(u-u_0) (dd^cu)^n >-\infty \ \hbox{is equivalent to} \ \int_Hu (dd^cu)^n >-\infty.$$
But this says that $\phi \in PSH(X,\omega)$ satisfies $\mu =MA(\phi)$; $\int_X MA(\phi) =1$; and $\int_X \phi MA(\phi) >-\infty$ -- that is, $\phi \in \mathcal E^1(X,\omega)$.

\begin{proposition} \label{gap} Let $\mu  \in {\mathcal M}(H)$ where $H$ is a compact set in $\CC^n$. Then
$$E^*(\mu)=\sup_{\tilde Q}[E(V^*_{H,\tilde Q})-\int_H \tilde Q d\mu]+\int_H u_0 d\mu$$
where the supremum is taken over all continuous $\tilde Q$ on $H$. 
\end{proposition}

\begin{proof} We first show that the right-hand-side coincides with
\begin{equation} \label{newe} \tilde E^*(\mu):= \sup [E(u)-\int_H (u-u_0) d\mu] \end{equation}
where the supremum is taken over all $u\in L(\CC^n)$ satisfying (\ref{nonpp}) and $ \int u(dd^cu)^n >-\infty$. Indeed, for an arbitrary $\mu  \in {\mathcal M}(H)$, from (\ref{newe}) and (\ref{suppw}) we have
$$\tilde E^*(\mu)\geq \sup_{\tilde Q}[E(V^*_{H,\tilde Q})-\int_H \tilde Q d\mu]+\int_H u_0 d\mu$$
where the supremum is taken over all continuous $\tilde Q$ on $H$. For the reverse inequality, fix $u\in L(\CC^n)$ with
 $$ \int_{\CC^n} (dd^cu)^n =1 \ \hbox{and} \  \int_{\CC^n}u (dd^cu)^n >-\infty.$$
From Remark \ref{decapprox}, we can construct continuous $u_j\in L^+(\CC^n)$ with $u_j \downarrow \tilde u\geq u$ such that $\tilde u = u$ on $H$ and $\mu_j:=(dd^cu_j)$ supported in $H$ for all $j$. In particular, $u_j = V_{H,Q_j}$ where $Q_j =u_j|_H$ is continuous. Thus, given $\epsilon >0$, we can choose $j$ sufficiently large so that, by monotone convergence,
$$-\int_H u_j d\mu \geq - \int_H u d\mu -\epsilon;$$
and, by (\ref{emono}), 
$$E(V_{H,Q_j})\geq E(u).$$
Hence
$$E(V_{H,Q_j}) - \int_H (u_j -u_0)d\mu \geq E(u)- \int_H (u-u_0) d\mu -\epsilon.$$
so that 
$$\sup_{\tilde Q}[E(V^*_{H,\tilde Q})-\int_H \tilde Q d\mu]+\int_H u_0 d\mu\geq \tilde E^*(\mu)$$
and equality holds.

To see that $E^*(\mu) =\tilde E^*(\mu)$, first assume $E^*(\mu) < +\infty$. By Corollary \ref{varprop}, $\mu =MA(\phi)$ for some $\phi \in \mathcal E^1(X,\omega)$ and 
$$E^*(\mu)= E(\phi)-\int_H \phi d\mu.$$
For the associated $u=\phi +u_0\in L(\CC^n)$ we have 
$$ \int_{\CC^n} (dd^cu)^n =\int_H (dd^cu)^n =\int_H MA(\phi)=1$$
and
$$\int_H(u-u_0) (dd^cu)^n =\int_H \phi MA(\phi) =\int_H \phi d\mu > -\infty$$
so that $\tilde E^*(\mu)\geq E(u)-\int_H (u-u_0) d\mu$. But
$$E(\phi)-\int_H \phi d\mu= E(u) - \int_H (u-u_0) d\mu$$
so $E^*(\mu) =\tilde E^*(\mu)$. If, on the other hand, $E^*(\mu) =+\infty$, Corollary \ref{varprop} provides the existence of $\phi \in \mathcal E^1(X,\omega)$ with $\phi \not\in L^1(\mu)$. By subtracting a constant we can assume $\phi \leq 0$ so that $\int_H \phi d\mu=-\infty$. Then for the associated $u=\phi +u_0\in L(\CC^n)$ we have $\int_H (u-u_0)d\mu =-\infty$; hence $\tilde E^*(\mu)=+\infty$.

\end{proof} 

The following corollary combines Proposition \ref{gap} with Theorem \ref{overthm}. Recall (\ref{torus}): $V_T(z) =\max_{j=1,...,n}\max [\log |z_j|, 0]$. 

\begin{corollary} \label{aftergap} Let $\mu  \in {\mathcal M}(H)$ where $H$ is a nonpluripolar, compact, convex set in $\CC^n$. Then
\begin{equation} \label{j=w=e} \log J(\mu)= \log W(\mu) =\inf_w [\log \bar \delta^w (H) +\int_H Qd\mu]\end{equation}
$$=-E^*(\mu)+\int_H u_0 d\mu +E(V_T)$$ 
where the infimum is taken over all  $w=e^{-Q}>0$ continuous on $H$. In particular, we have $E^*(\mu)=+\infty$ if and only if $J(\mu)=W(\mu)=0$.
\end{corollary}

\begin{proof}  First of all, from results of Berman and Boucksom in \cite{BBnew}, 
$$ -\log  \delta^w (H)=(\frac{n+1}{n}) [E(V^*_{H,Q})-E(V_T)]; \ \hbox{i.e.},$$ 
$$-\log  \bar \delta^w (H)= E(V^*_{H,Q})-E(V_T).$$
Thus we get from Theorem \ref{overthm} and Proposition \ref{gap} that
$$\log J(\mu)= \log W(\mu) =  \inf_w [\log \bar \delta^w (H) +\int_H Qd\mu]$$
$$=-E^*(\mu)+\int_H u_0 d\mu +E(V_T)> -\infty$$ 
for all $\mu  \in {\mathcal M}(H)$ for which there exists $u\in L(\CC^n)$ satisfying (\ref{nonpp}) with $(dd^cu)^n=\mu$ and $\int_H ud\mu > -\infty$. From Corollary \ref{varprop}, to complete the proof we must show that if $E^*(\mu)=+\infty$ then $J(\mu)=W(\mu)=0$. This follows from Proposition \ref{gap} and Proposition \ref{upper}: 
$$\log J(\mu) \leq \log W(\mu)\leq \inf_w [\log \bar \delta^w (H) +\int_H Qd\mu]$$
$$=-E^*(\mu)+\int_H u_0 d\mu +E(V_T).$$
\end{proof}

From Corollary \ref{overcor} and Corollary \ref{aftergap}, we also have the following.

\begin{corollary}\label{finalcor} Let $\mu  \in {\mathcal M}(H)$ where $H$ is a nonpluripolar, compact, convex set in $\CC^n$ and let $Q$ be continuous on $H$. Then
\begin{equation}\label{minwtd}\log J^Q(\mu)=\log W^Q(\mu)= -E^*(\mu)-\int_H (Q-u_0)d\mu +E(V_T).\end{equation}
\end{corollary}

We conclude this section with the statement and proof of a version of the domination principle needed for the proof of Proposition \ref{goodapprox}. This is due to S. Dinew.

\begin{proposition} \label{domprin} Let $\psi \in PSH(X,\omega)$ and $\phi \in \mathcal E(X,\omega)$ satisfy $\psi \leq \phi$ a.e.-$(dd^c\phi +\omega)^n$. Then $\psi \leq \phi$ on $X$.
\end{proposition}

\begin{proof} By \cite{[GZ2]} Proposition 1.6,  $\max(\phi, \psi)\in \mathcal E(X,\omega)$; similarly, fixing $\epsilon >0$, $\max(\phi +\epsilon, \psi)\in \mathcal E(X,\omega)$. We have
$$\bigl(dd^c \max(\phi +\epsilon, \psi) +\omega \bigr)^n\geq {\bf 1}_{\{\phi +\epsilon> \psi\}}\bigl(dd^c \max(\phi +\epsilon, \psi) +\omega \bigr)^n$$
$$\geq {\bf 1}_{\{\phi +\epsilon> \psi\}}\bigl(dd^c (\phi +\epsilon) +\omega \bigr)^n={\bf 1}_{\{\phi +\epsilon> \psi\}}\bigl(dd^c \phi +\omega \bigr)^n$$
where the inequality in the previous line uses Corollary 1.7 of \cite{[GZ2]}. From Theorem 1.9 of \cite{[GZ2]},
\begin{equation} \label{romanone}
\bigl(dd^c \max(\phi +\epsilon, \psi) +\omega \bigr)^n\to \bigl(dd^c \max(\phi , \psi) +\omega \bigr)^n\end{equation}
weak-* as $\epsilon \to 0$; while
$${\bf 1}_{\{\phi +\epsilon> \psi\}}\bigl(dd^c \phi +\omega \bigr)^n\to {\bf 1}_{\{\phi \geq \psi\}}\bigl(dd^c \phi +\omega \bigr)^n$$
weak-* as $\epsilon \to 0$. By hypothesis, 
$${\bf 1}_{\{\phi \geq \psi\}}\bigl(dd^c \phi +\omega \bigr)^n=\bigl(dd^c \phi +\omega \bigr)^n.$$

We claim that these limiting measures satisfy
\begin{equation} \label{claimindom}
\bigl(dd^c \max(\phi , \psi) +\omega \bigr)^n\geq \bigl(dd^c \phi +\omega \bigr)^n.
\end{equation}
Assuming (\ref{claimindom}), since 
$$\int_X \bigl(dd^c \max(\phi , \psi) +\omega \bigr)^n =\int_X \bigl(dd^c \phi +\omega \bigr)^n=1,$$
we conclude that $\bigl(dd^c \max(\phi , \psi) +\omega \bigr)^n= \bigl(dd^c \phi +\omega \bigr)^n$. By the uniqueness result in \cite{[D]}, $\max(\phi , \psi) -\phi$ is a constant, which must be zero by hypothesis.

Finally, to prove (\ref{claimindom}), it suffices to show that if $f\in C(X)$ with $f\geq 0$, then
$$\int_X f \bigl(dd^c \max(\phi , \psi) +\omega \bigr)^n\geq \int_X f \bigl(dd^c \phi +\omega \bigr)^n.$$
By (\ref{romanone}), 
$$\int_X f \bigl(dd^c \max(\phi , \psi) +\omega \bigr)^n=\lim_{\epsilon \to 0} \int_X f \bigl(dd^c \max(\phi +\epsilon, \psi) +\omega \bigr)^n$$
$$\geq \lim_{\epsilon \to 0} \int_X f \cdot {\bf 1}_{\{\phi +\epsilon> \psi\}}\bigl(dd^c \phi +\omega \bigr)^n=
\int_X f \bigl(dd^c \phi +\omega \bigr)^n.$$
\end{proof}

\noindent Writing $u,v\in L(\CC^n)$ as $u=\phi +u_0,v=\psi +u_0$ with $\phi,\psi\in PSH(X,\omega)$, we have the $\CC^n-$version of the domination principle.

\begin{corollary}\label{domcorr} Let $u,v\in L(\CC^n)$ and suppose $u$ satisfies (\ref{nonpp}). If $v\leq u$ a.e-$(dd^cu)^n$, then $v\leq u$ in $\CC^n$.
\end{corollary}

\begin{remark} 
Corollary \ref{domcorr} generalizes not only the multivariate ($n>1$) domination principle of Bedford-Taylor (Lemma 6.5 in \cite{[BT]}), but also the univariate ($n=1$) version, Theorem 3.2 in Chapter II of  \cite{ST}.
\end{remark}

\section{Final comments.}\label{sec:finalcom}

We end with some comments. In the univariate setting, it was proved in \cite{bloomvoic} that, using our notation, 
$$-\log W(\mu)=\frac{1}{2}I(\mu):=\frac{1}{2}\int_H \int_H \log \frac{1}{|z-t|} d\mu(t)d\mu(z),$$
the logarithmic energy of $\mu$, and, more generally,
$$-\log W^Q(\mu)=\frac{1}{2}I_Q(\mu):=\frac{1}{2}\int_H \int_H \log \frac{1}{|z-t|w(z)w(t)} d\mu(t)d\mu(z)$$
$$= \frac{1}{2}I(\mu)+\int_HQd\mu,$$
the weighted logarithmic energy of $\mu$. The factor $1/2$ comes from our use of the normalized weighted transfinite diameters $\bar \delta^w(H)= \delta^w(H)^{1/2}$. The equilibrium measures $\mu_{eq}(H)$  and $\mu_{eq}(H,Q)$ minimize $I(\mu)$ and $I_Q(\mu)$ over all $\mu\in {\mathcal M}(H)$. In the higher dimensional case, as we have noted, our functional $-\log W(\mu)$ is related to the electrostatic energy $E^*(\mu)$ defined by Berman, Boucksom, Guedj and Zeriahi in their seminal paper \cite{[BBGZ]}. In that paper, a special case of their Theorem 5.3 corresponding to the setting described in our Section \ref{sec:app} shows that $\mu_{eq}(H,Q)$ mimimizes the functional 
\begin{equation} \label{wtden}E^*(\mu) +\int_H (Q-u_0)d\mu = -\log W^Q(\mu) +E(V_T)\end{equation}
(see (\ref{bbgzfunc}) and (\ref{minwtd})) over all $\mu\in {\mathcal M}(H)$, generalizing the univariate result. Thus from Corollaries \ref{aftergap} and \ref{finalcor} one can interpret $$-\log W(\mu) \ \hbox{and}   \ -\log W^Q(\mu)=-\log W(\mu)+\int_H Qd\mu$$ 
as energies (weighted energies) to be minimized in the pluripotential theoretic setting; i.e., {\it pluripotential energies}.

Our definition of $W(\mu)$ essentially involves only the support of $\mu$ and shows that for a ``good'' sequence of discrete approximations $\mu_s:=\frac{1}{s}\sum_{j=1}^s \delta (a_j)\to \mu$ weak-*, we have
$$|VDM_d({\bf a})|^{1/ds} \to W(\mu).$$
In particular, we have $W(\mu)=0$ if and only if for {\bf any} array of points $\{A_n\}$ in $H$ where 
$A_n = \{a_n^1,...,a_n^{s_n}\}$ is a set of $s_n=s_n(d_n)$ points in $H^{s_n}$ with $d_n\uparrow \infty$ such that $\mu_n:= \frac{1}{s_n} \sum_{j=1}^{s_n} \delta(a_n^j)\to \mu$ weak-* we have
$$|VDM_{d_n}(A_n)|^{1/d_ns_n}\to 0.$$
Thus although our definitions involve an underlying compact set $H$, given a probability measure $\mu$ with compact support, information about $\mu$ near its support is sufficient to detect positivity of its pluripotential energy. Indeed, for $\mu$ a probability measure with compact support, Corollary \ref{aftergap} shows that for {\it any} convex, compact nonpluripolar set $H$ containing supp$(\mu)$ (or even any fat, connected, subanalytic set $H$; cf., Remark \ref{independ}), the functionals $J(\mu)$ and $W(\mu)$ defined relative to $H$ agree and coincide with $\exp(-E^*(\mu))$ up to the term $\int_Hu_0d\mu+E(V_T)$. Since $\int_Hu_0d\mu$ is the same for any such $H$, our definitions are independent of which set $H$ we choose.

Finally, a comment on the density condition: we noted that this hypothesis on $\nu$ implies that $(H,\nu)$ satisfies a Bernstein-Markov property for holomorphic polynomials on $\CC^{n}$; indeed, $(H,\nu)$ satisfies a Bernstein-Markov property for {\it real} polynomials on $\RR^{2n}$. Following the proof of Proposition 3.1 in \cite{bloomvoic}, for any positive, continuous weight $w$ on $H$, the triple $(H,\nu,w)$ satisfies a {\it weighted Bernstein-Markov property} for holomorphic polynomials on $\CC^{n}$: for all $p_d\in \mathcal P_d$, 
\begin{equation}\label{wtdbernmark}||w^dp_d||_H \leq M_d \bigl(\int_H |p_d(z)|^2  |w(z)|^{2d} d\nu(z) \bigr)^{1/2}\ \hbox{with} \ \limsup_{d\to \infty} M_d^{1/d} =1.\end{equation}
In \cite{[Ber]}, such measures $\nu$ were called {\it strongly Bernstein-Markov} on $H$.

\end{document}